\documentclass[12pt,a4paper]{article}
\usepackage[latin1]{inputenc}
\usepackage[english]{babel}
\usepackage[left=3cm,right=2.5cm,top=3cm,bottom=4cm]{geometry}
\usepackage{upquote}
\usepackage{upgreek}
\usepackage{enumerate}
\usepackage{subfigure}
\usepackage{booktabs}
\usepackage{multirow}
\usepackage{longtable}
\usepackage{color,soul}
\RequirePackage[OT1]{fontenc}
\RequirePackage{amsthm,amsmath,amssymb,amscd,amsfonts,latexsym, enumerate}
\RequirePackage[numbers]{natbib}
\RequirePackage[colorlinks,citecolor=blue,urlcolor=blue]{hyperref}
\usepackage{graphicx}
\numberwithin{equation}{section}
\theoremstyle{plain}

\newtheorem{theorem}{Theorem}[section]

\newtheorem{corollary}{Corollary}[section]

\newtheorem{lemma}{Lemma}[section]

\def \a{\alpha}
\def \ua{\underline{\alpha}}
\def \u{\bar{u}}
\def \ur{\underline{r}}
\def \br{\bar{r}}

\def \v{\varphi}
\def \vn{\varphi_n}
\def \e{\epsilon}
\def \be{\bar{\epsilon}}

\usepackage{authblk}
\newcommand*{\email}[1]{%
    \normalsize\href{mailto:#1}{#1}\par
    }
\begin{document}
\title{Estimating the logarithm of
characteristic function and stability
parameter for symmetric stable laws}

\author{J\"uri Lember\thanks{Estonian institutional research funding IUT34-5, Estonian Research Council grant PRG865}}
\affil{\normalsize Institute of Mathematics and Statistics,
University of Tartu,  Estonia,\\\email{jyril@ut.ee}}

\author{Annika Krutto\thanks{Estonian institutional research funding IUT34-5}}
\affil{\normalsize Institute of Mathematics and Statistics, University of Tartu,  Estonia\\
 Department of Biostatistics, University of Oslo, Norway,\\\email{krutto@ut.ee}} 

\setlength\parindent{0pt}
\setlength{\parskip}{1em}
\date{\vspace{-5ex}}
\maketitle \thispagestyle{empty}
\begin{abstract}
Let $X_1,\ldots,X_n$ be an i.i.d. sample from symmetric stable distribution with stability parameter $\alpha$ and scale parameter $\gamma$. Let $\varphi_n$ be the empirical characteristic function.   We prove an uniform large deviation inequality: given preciseness $\epsilon>0$ and probability $p\in (0,1)$, there exists universal (depending on $\e$ and $p$ but not depending on $\a$ and $\gamma$) constant $\br>0$ so that
$$P\big(\sup_{u>0:r(u)\leq \br}|r(u)-\hat{r}(u)|\geq \e\big)\leq p,$$
where $r(u)=(u\gamma)^{\a}$ and $\hat{r}(u)=-\ln|\varphi_n(u)|$. As an applications of the result, we show how it can be used in estimation unknown stability parameter $\a$.
\end{abstract}
\textbf{Keywords}: Stable laws, large deviation inequalities, parameter estimation;

\textbf{MSC codes}: Stable laws {60E07}; Confidence regions {62F25}

\section{Introduction and preliminaries}
{Let $X_1,\ldots,X_n$ be an i.i.d. sample from   stable law with  characteristic function
$\varphi(u)=\exp\left[-(\gamma|u|)^{\a}+i\omega(u)\right]$,
where $\omega(u)=u[\beta \gamma \tan \frac {\pi \alpha}{2} (|\gamma u|^{\alpha-1}-1)+\delta]$ for $\alpha \neq 1$ and $\omega(u)=u[-\beta  \gamma \frac {2}{\pi}
    \ln(\gamma|u|)+\delta]$  for $\alpha=1$,  and $\alpha \in (0,2]$, $\beta \in [-1,1]$, $\gamma > 0$, $\delta \in \mathbb{R}$ are (unknown) stability, skewness,  scale and shift parameters, respectively. } Basic properties of stable distributions can be found in \cite{zolotarev1986,Samorodnitsky1994,Nolan2018}. Let $F_n$ be the  empirical distribution function, and
 $\varphi_n$ the empirical characteristic function, i.e.
\begin{equation}\label{emp}
 \v_n(u)=\int_{\mathbb{R}}\exp\{iux\}\operatorname{d}F_n(x),\quad u\in \mathbb{R}.
 \end{equation}  Let  $r(u)=-\ln|\varphi(u)|=(\gamma|u|)^{\a}$ and $\hat{r}(u)=-\ln|\varphi_n(u)|$. {Estimating the parameters of  stable law  is a notoriously hard problem} (see, e.g.,  \cite[Section 2]{Nolan2001}, \cite[Chapter 4]{zolotarev1986}).
 Simple empirical characteristic function based {closed form} estimates were proposed in \cite{Press1972}. In particular, the stability parameter estimator is
 \begin{equation}\label{hinn0}
\hat{\a}={\ln \hat{r}(u_1)-\ln \hat{r}(u_2)\over \ln u_1-\ln u_2},
\end{equation}
where $0<u_2<u_1$ are fixed arguments and $\hat{r}(u)=-\ln|\vn(u)|$.
Since for any $u$,   $\hat{r}(u)\to (\gamma|u|)^{\a}$, a.s., we see
that any choice of  {$0<u_2<u_1$} gives consistent estimator. The same points $0<u_2<u_1$ can be used to give consistent  closed form estimators also to other parameters $\gamma$, $\beta$, $\delta$ of stable law
(see \cite[Theorem 1]{Krutto2016}). Despite the asymptotic consistency holds for any pair $u_1,u_2$, in
practice the right choice of $u_1$ and $u_2$ is crucial and the universal selection  of these values has remained unsolved (e.g., \cite{Paulson1975,Bibalan2017CharacteristicFB, kakinaka2020flexible}).  Recently, \cite{Krutto2018} {suggests that the  choice of $u_1$ and $u_2$ should be fixed based on the preciseness of $\hat{\a}$, that is, on the preciseness of $\hat{r}(u)$.} Clearly, the
preciseness of $\hat{\a}$ depends on how well  $\hat{r}(u_i)$ and
estimates $r(u_i)$ for $i=1,2$.
\begin{figure}[!htp] \centering{\includegraphics[height=0.25\textheight, width=1\textwidth]{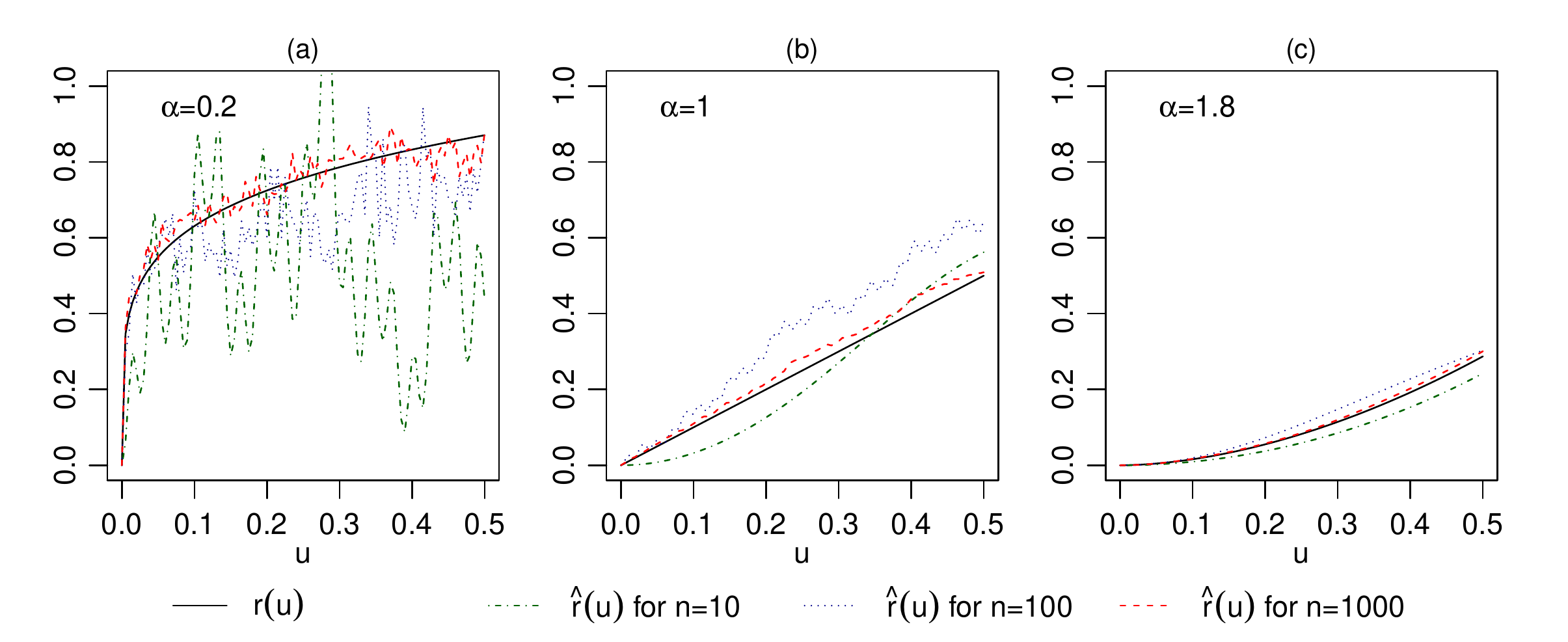}}\caption{{The values of $r(u)=u^\a$ vs $\hat{r}(u)$ of single replicates (simulated with} \cite{STABLER}) of stable law with $\gamma=1$, $\delta=0$, $\beta=0$ for $\a=0.2$ in (a), $\a=1$ in (b) and $\a=1.8$ in (c) .}\label{fig1}
\end{figure}
\begin{figure}[htp]
\centering
\includegraphics[height=0.25\textheight, width=1\textwidth]{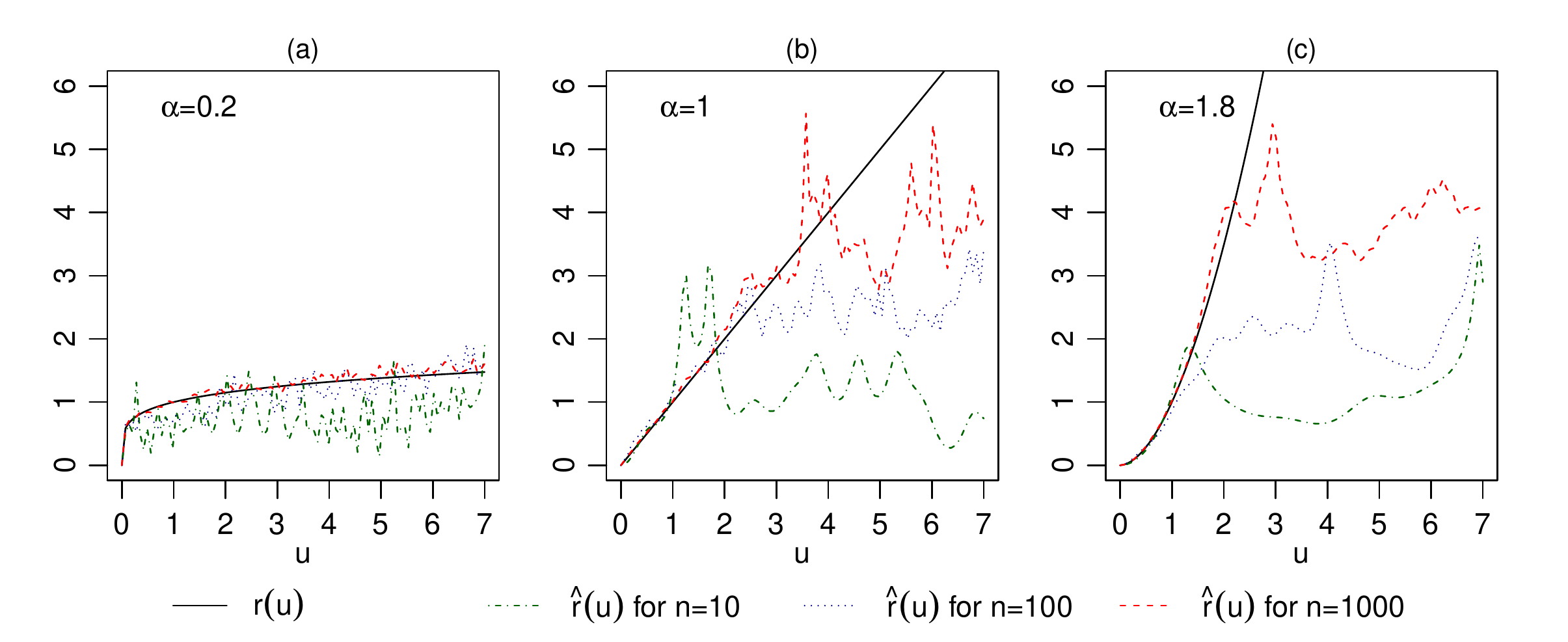}\caption{The values of $r(u)=u^\a$ vs $\hat{r}(u)$ of single replicates (simulated with \cite{STABLER}) of stable law with $\gamma=1$, $\delta=0$, $\beta=0$ for $\a=0.2$ in (a), $\a=1$ in (b) and $\a=1.8$ in (c) .}\label{fig2}
\end{figure}

Figure \ref{fig1} and Figure \ref{fig2} show that
$\hat{r}(u)$ is relatively accurate estimate of $r(u)$ only in a small
interval $(0,\u]$, where $\u$ obviously depends on sample size $n$,
but unfortunately also on $\alpha$ -- the smaller $\a$, the smaller
also $\u$. So, there seems not to exist an universal (that applies for any $\a\in (0,2]$)  upper bound $\u$ so that
$\sup_{u\in (0,\u]}|\hat{r}(u)-r(u)|$ were small for any $\alpha$ even when $\gamma>0$ is known.  It is clearly evident from {Figure} \ref{fig1} and Figure \ref{fig2}, that too
big $u_2$ makes the estimate of $\a$ very imprecise. Observe that {$0<u_2<u_1$} cannot also be very small, because then  {$\ln u_1-\ln u_2$} is also very small  and that affects the preciseness of $\hat{\a}$ even when $\hat{r}(u_i)\approx r(u_i)$, $i=1,2$.  On the other hand, Figure \ref{fig1} and Figure \ref{fig2} as well as simulations in \cite{Krutto2018} suggest that despite the possible non-existence of universal $\u$, there might exists an universal (not depending on $\alpha$ and $\gamma$) $\br$ so that
$\sup_{0<u: r(u)\leq \br}|{\hat{r}}(u)-r(u)|$
is relatively small.

Our main theoretical result states that for symmetric \footnote{{For mathematical tractability in formulas, in particular in tail estimation in} \eqref{LL} and \eqref{Ka}, we provide  our results for symmetric stable laws. Similar construction of proof can be applied for general stable laws.}
stable laws such universal $\br$ exists.

\begin{theorem}\label{thm} Let $X_1,\ldots,X_n$ be an i.i.d. sample from symmetric stable law. Fix $\e>0$, $p\in (0,1)$. Then there exists $n_o(\e,p)<\infty$ such that
 for every $\alpha\in (0,2]$ and $n>n_o$
\begin{equation}\label{v}
P\Big( \sup_{0<u: r(u)\leq
\br}|r(u)-\hat{r}(u)|>\e\Big)\leq p,\end{equation} where
$\bar{r}(\e,p,n)>0$ is independent of $\a$ and $\gamma$.\end{theorem}
Theorem is proved in Subsection \ref{sec:main}. From  Theorem \ref{thm}
it follows that with probability $1-p$, $|r(u_i)-\hat{r}(u_i)|\leq
\epsilon$, given $n$ is big enough and $u_i$ is cosen such that $r(u_i)\leq \br$, $i=1,2$.
Therefore, in order to apply (\ref{hinn0}), it makes sense not to fix arguments $u_1$ and $u_2$, but
the values $\hat{r}(u_1)$ and $\hat{r}(u_2)$ instead. The obtained
estimate is then
\begin{equation}\label{hinn}
\hat{\a}={\ln (\br-\e)-\ln (\ur+\e) \over \ln u_1-\ln u_2},\end{equation}
where, $\e$, $\br$ and $\ur$ are carefully chosen constants and
\begin{equation}\label{u}
u_1=\inf\{u: \hat{r}(u)=\br-\e\},\quad u_2=\sup\{u:
\hat{r}(u)=\ur+\e\}.
\end{equation}
The estimates of parameters $\a, \beta, \gamma,\delta$ at $u_1$ and $u_2$ based on $\br-\e=0.5$ and $\ur+\e=0.1$ were
proposed and studied in \cite{Krutto2018}. The current article thus provides theoretical justification to the such method of  argument selection.

{Note that} estimating the parameters via two points $u_1$ and $u_2$ as in \eqref{hinn0}  deserves more attention in the
recent literature and specifying $u_1$ and $u_2$ via $\hat{r}$ function as in \eqref{u} is not
the only option. In \cite{Bibalan2017CharacteristicFB}, $u_1=1 $ and  $u_2$ is taken such that the distance between
 Cauchy ($\alpha=1$) and Gaussian ($\alpha=2$) characteristic functions at $u_2$ were maximal.
The idea of maximizing the distance between characteristic functions is further developed in
\cite{kakinaka2020flexible}, where an iterative 8-step algorithm for specifying $u_1$ and $u_2$ is proposed.
Although the  idea of maximizing discrepancy between characteristic functions (or rather
between $r$-functions) is quite natural, and  the simulations in \cite{kakinaka2020flexible} show good behaviour
 of that choice in practice,
the proposed algorithm in \cite{kakinaka2020flexible} is still ad hoc in nature and lacks theoretical
justification.  In particular, it is not clear that it allows to choose $u_1$ and $u_2$ so that the
inequality \eqref{pr}  below holds.

The proof of the
Theorem \ref{thm} is constructive, but the goal of it is to  show
that universal $\br$ exists, not to optimize the constant, i.e. to
find the biggest possible $\br$ and smallest possible $n_o$. It
means that the $\br$ constructed in the proof is probably too small
for practical use. Although, for every $u$, $\gamma$ and $\a$,
$\hat{r}(u)\to r(u)$, a.s. as $n$ grows, our upper bound satisfies
$\br(\e,p,n)\leq w_o<1$, where $w_o$ is a constant depending on
$\e$. Thus $\br$ is always bounded away from 1 and
does not increase to the infinity as $n$ grows. This need not
necessarily be the deficiency of the proof, rather than necessary
property. To see that, assume  the inequality (\ref{v}) holds with
some $\br>1$. Then it follows that
$$P\Big( \sup_{u\in (0,\u]}|r(u)-\hat{r}(u)|>\e\Big)\leq p,$$
where $\u=\sqrt{\br \over \gamma}$, so that the universal upper bound (that applies for any $\a$)
$\u$ would exists. However, there is no evidence at all that such an
$\u$ exists, suggesting that the a.s. convergence $\hat{r}(u)\to
r(u)$ is not uniform over small $\alpha$ -- for every $u$ and $n$
there exists $\alpha$ small enough so that the difference
$|r(u)-\hat{r}(u)|$ is still big. If so, then $\br$ must
always remain smaller than 1.
The situation is different, when we bound the unknown stability
parameters $\a$ below from zero, i.e. we assume the existence of
$\ua>0$ (which can be arbitrary small) such that the unknown parameter $\a$ belongs to $[\ua,2]$.
In this case the upper bound $\br(\e,p,n,\ua)$ satisfies $\lim_n
\br(\e,p,n,\ua)=\infty$ (Corollary \ref{kolla}), and in this case also
the upper bound $\u$ exists. The existence of $\ua$ is common
assumption in practice ({e.g.,}  \cite{McCulloch1986,Kogon1998,Nolan2001} {suggest} $\ua=0.5$) and we keep this additional
assumption in Theorem \ref{thm2}, that provides a uniform bound
similar to (\ref{v}) to the difference of logarithms $|\ln
\hat{r}(u)-\ln r(u)|$. Theorem \ref{thm2} is actually a simple
corollary of Theorem \ref{thm}, but the additional assumption about
the existence of $\ua$ is necessary, because  if $r$ is very small,
then $|\ln \hat{r}-\ln r|$ can be rather big even when $r$ is very
close to $\hat{r}$.
\\\\
Section \ref{sec:hinnang} is devoted to the applications of Theorem
\ref{thm2} in the light of large deviation inequalities for
$\hat{\a}$. We show how the upper bound from
Theorem \ref{thm2} can be used to solve the two basic questions
related with estimate $\hat{\a}$  in (\ref{hinn}):
\begin{itemize}
  \item  Given
precision $\e_1>0$, probability $p$ and lower bound $\ua$, find
$\br,\ur,\e$ (needed to construct $\hat{\a}$ in (\ref{hinn})) and
possibly small sample size $n$ so that
\begin{equation}\label{pr}
P\big(\mid \hat{\a}-\a\mid> \e_1\big)\leq p.
\end{equation}
  \item Given sample size $n$ and $\ua$, find $\br,\ur,\e$ and possibly
small $\e_1$ so that (\ref{pr}) holds. In other words, find exact
i.e. non-asymptotic confidence interval to $\a$.
\end{itemize}
 The solutions of these questions are formulated as Theorem
\ref{thm21} and Theorem \ref{thm22}.\\\\
The bound in Theorem \ref{thm2} is constructed using the basic bound
$\br(\epsilon,p,n)$ provided by Theorem \ref{thm}. Thus the  $\e_1$
and required sample size $n$ in the inequality (\ref{pr}) depend
heavily on the function $\br$. Unfortunately, the function $\br$ constructed in Section \ref{sec:main} is not explicitly given and, hence, difficult to work with. Although the main goal of the
present paper is just to show that the function $\br$ exists,  in Section \ref{sec:alternative}, we discuss another possibility to
construct $\br$. The new construction gives analytically more tractable bound, the price for it is bigger minimal required
sample size $n_o$ and lower bound.  Also the upper bound $\br$ constructed in Section  \ref{sec:alternative} is also strictly smaller than 1 for every $n$. So we have two different constructions
with the same property, and this allows us to conjecture that even the best bound $\br$ is always smaller than 1 and we also conjecture that the universal upper bound $\u$ does not exist.

\paragraph{Preliminaries.}

For $0<\alpha<2$ every stable distribution has tail(s) that are asymptomatically power laws with heavy tails (e.g., \cite[Theorem 1.12]{Nolan2018}, \cite[Property 1.2.15]{Samorodnitsky1994},
\begin{align*}
F(t)&\sim  c_\a\gamma^\alpha(1-\beta) |t|^{-\alpha},\quad \text{when }t\to -\infty, \\
1-F(t)&\sim c_\a\gamma^\alpha (1+\beta)  t^{-\alpha},\quad \text{when }t\to \infty,\quad
\end{align*}
where $c_\a=\frac{\Gamma(\alpha)}{\pi}\sin\frac{\pi\alpha}{2}\leq\frac{1}{2}$, $ \lim_{\alpha\to 0}c_\a=\frac{1}{2}$, and $ \lim_{\alpha\to 2}c_a=0.$
For {symmetric stable laws ($\beta=0$)} these results imply the existence of
constants\footnote{{For general stable laws it implies for the exitsence of constants $L_1(\a,\beta)=L(a)(1-\beta)$ and $L_2(\a,\beta)=L(a)(1+\beta)$ with$(1-\beta)\in [-2,0]$ and $(1+\beta)\in [0,2]$.}} $L(\a)$ so that
\begin{equation}\label{LL} F(t)\leq
L(\alpha)\left({|t|\over \gamma}\right)^{-\alpha},\quad \forall t<0,\quad
(1-F(t))\leq L(\alpha)\left({t\over \gamma}\right)^{-\alpha},\quad \forall
t>0,\end{equation}
 Figure \ref{fig3} plots the function
$t\mapsto F(t){|t|}^{\alpha}$ for different $\alpha$-values and $\gamma=1$ in the range $[-5,0]$ and $[-10^7,0]$.  Increasing the interval shows similar pattern, hence it is clear that there exists $L<\infty$ so that $\sup_{\a\in (0,2]}L(\a)\leq L$. Throughout the paper, we keep $L$ undetermined, although one can
take it as ${1\over 2}$. It is also obvious that $L$ is independent of $\gamma$.
\begin{figure}[htp]
\centering
\includegraphics[height=0.25\textheight, width=1\textwidth]{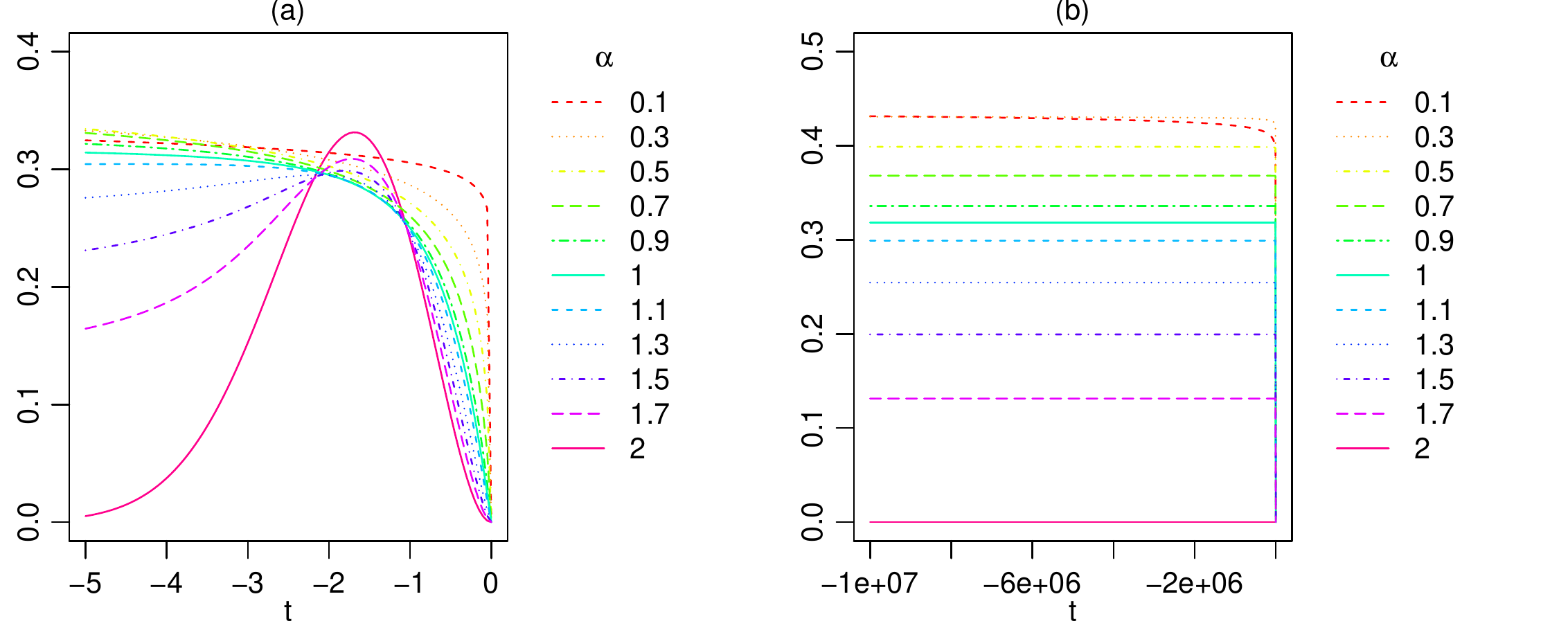}\caption{{Plotting
$t\mapsto F(t){|t|}^{\alpha}$ (calculated with} \cite{STABLER})for different $\alpha$-values with $\gamma=1, delta=0, \beta=0$ in the range $[-5,0]$ and $[-10^7,0]$ .}\label{fig3}
\end{figure}

In what follows, we shall use the following elementary inequalities: for any $x,y>0$,
$$|\ln x-\ln y|\leq {|x-y|\over x\wedge y}={|x-y|\over x}\vee
{|x-y|\over x-|x-y|},$$ we obtain that $|\ln x-\ln y|>\e$ implies
${|x-y|\over x}>{\e\over 1+\e}$. Thus, for any $\u>0$ and $\e>0$
\begin{equation}\label{vorr1}
P\big( \sup_{u\in (0,\u]}|r(u)-\hat{r}(u)|>\e\big)\leq P\Big(
\sup_{u\in (0,\u]}e^{r(\u)}|\varphi(u)-\v_n(u)|>{\e\over 1+\e}\Big).
\end{equation}
Observe that (\ref{vorr1}) holds for any estimate $\v_n$.
\section{Main results}\label{sec:main}
\subsection{Proof of Theorem \ref{thm}}
\label{subsec1:main}
\paragraph{Bounding the difference of characteristic functions.} Recall
$\v_n$ is  the standard empirical estimate {given by} (\ref{emp}). To
bound $|\varphi(u)-\v_n(u)|$ we use the approach in
\cite{Csorgo1981} as follows. For every $0<K<\infty$,
\begin{align*}
|\v_n(u)-\v(u)|&=|\int e^{iux}d(F_n(x)-F(x))|\leq |\int_{-\infty}^{-K} e^{iux}d(F_n(x)-F(x))|\\
& +|\int_{-K}^{K} e^{iux}d(F_n(x)-F(x)) |+|\int_{K}^{\infty}
e^{iux}d(F_n(x)-F(x))|.
\end{align*}
Since $|e^{iux}|=1$ for every $u$ and $x$, the first term can be bounded
\begin{align*}
|\int_{-\infty}^{-K} e^{iux}d(F_n(x)-F(x))| &\leq
\int_{-\infty}^{-K} |e^{iux}|dF_n(x)+\int_{-\infty}^{-K}
|e^{iux}|dF(x)\leq F_n(-K)+F(-K)\\
&\leq \|F_n-F\|+2F(-K),
\end{align*}
where $\|F_n-F\|:=\sup_x|F_n(x)-F(x)|$. The last inequality holds, because  $F_n(-K)\leq  \|F_n-F\|+F(-K)$.  Similarly, the third term can be estimated above by $\|F_n-F\|+2(1-F(K))$. To estimate the second term, we use the integration by parts
\begin{align*}
|\int_{-K}^{K} e^{iux}d(F_n(x)-F(x)) |&=\Big|e^{iux}(F_n(x)-F(x))|_{-K}^K-iu \int_{-K}^K e^{iux}
(F_n(x)-F(x))dx\Big|\\
&\leq |F_n(-K)-F(-K)|+|F_n(K)-F(K)|+|u|\int_{-K}^K|F_n(x)-F(x)|dx\\
&\leq 2\|F_n-F\|+|u|\cdot \|F_n-F\|\cdot 2K.
\end{align*}
Therefore, for any $K$
\begin{equation}\label{3osa}
|\v_n(u)-\v(u)|\leq 4\|F_n-F\|+|u|\cdot \|F_n-F\|\cdot 2K+2F(-K)+2(1-F(K)).
\end{equation}
For any $\delta>0$, let $K(\delta)$ be so big that $F(-K)\leq {\delta\over 8}$. Then also $1-F(K)\leq {\delta\over 8}$ and (\ref{3osa}) implies
\begin{equation}\label{3osa2}
|\v_n(u)-\v(u)|\leq {\delta\over 2}+4\|F_n-F\|+|u|\cdot \|F_n-F\|\cdot 2K(\delta).
\end{equation}
By (\ref{LL}), we can take
\begin{equation}\label{Ka}
K(\delta)=\Big({ 8L\over \delta }\Big)^{1\over \alpha}{\gamma}
\end{equation}
and so with $u>0$, by (\ref{3osa2}) and (\ref{Ka})
\begin{equation}
|\v_n(u)-\v(u)|\leq {\delta\over 2}+2\|F_n-F\|\Big(2+\big({8Lr(u)\over \delta}\big)^{1\over \a}\Big).
\end{equation}
\paragraph{Bounding $r$.} Recall $\br=r(\u)$. Fix $\u>0$ and define $\be=\e e^{\br}$. We now use Dworetzky-Kiefer-Wolfowitz inequality
\cite[p.\ 268]{van1998asymptotic}):
$$P(\|F_n-F\|>\e)\leq 2\exp[-2n\e^2]$$
to estimate
\begin{align*}
P\Big(e^{r(\u)}\sup_{u\leq
\u}|\v_n(u)-\v(u)|>\e\Big)&\leq P\Big(\|F_n-F\|\geq {(\be-\delta/2)\delta^{1\over \a}\over
2\big(2\delta^{1\over \a}+(8L\br)^{1\over \a}\big)}\Big)\\
&= P\Big(\|F_n-F\|\geq {2^{1\over \a}(\be-\delta/2)({\delta\over 2})^{1\over \a}\over
2\big(2^{1+{1\over \a}}({\delta\over 2})^{1\over \a}+(8L\br)^{1\over \a}\big)}\Big)\\
&= P\Big(\|F_n-F\|\geq {(\be-\delta/2)({\delta\over 2})^{1\over \a}\over
2\big(2({\delta\over 2})^{1\over \a}+(4L\br)^{1\over \a}\big)}\Big)\\
&\leq 2\exp\Big[-{n\over 8}\cdot \Big( {(\be-\delta/2)({\delta\over 2})^{1\over \a}\over
({\delta\over 2})^{1\over \a}+{1\over 2}(4L\br)^{1\over \a}}\Big)^2\Big].
\end{align*}
 Define
$$k(\a,\e,\br)={1\over 8}\Big[\max_{0\leq x\leq \be}{(\be-x)\over 1+{1\over 2}\big({4L\br\over x}\big)^{1\over \alpha}}\Big]^2.$$
From (\ref{vorr1}), we obtain
$$
P\big(\sup_{u\in (0,\u]}|r(u)-\hat{r}(u)|>\e\big)\leq 2 \exp[-n
k\big(\a,{\e\over 1+\e},{\bar r}\big)   ]=p$$ which is equivalent to $
k\big(\a,{\e\over 1+\e},{\bar r}\big)={\ln(2/p)\over
n}$, and so the desired upper bound for any $\a$,
denoted by $r_n(\a)$ is the solution of the following equality
\begin{equation}\label{r3def}
k\big(\a,{\e\over 1+\e},r_{n}\big)={\ln(2/p)\over n}.
\end{equation}
Observe that
$$\lim_{r\to 0}k\big(\a,{\e\over 1+\e},{r}\big)={1\over 8}\big({\e\over 1+\e}\big)^2.$$
Hence the following condition gives a lower bound for minimal sample size $n$ so that $r_{n}(\a)>0$:
\begin{equation}\label{tingimus3}
n>8 \ln(2/p)\big({1+\e\over\e}\big)^2.
\end{equation}
\paragraph{The existence of $\br=\inf_{\a\in (0,2]} r_n(\a)$.}
The following lemma shows that $\inf_{\a\in (0,2]}r_{n}(\a)>0$, hence the universal (not depending on $\a$ and $\gamma$) bound $\br$
exists. Since $\e/(1+\e)<1$, without loss of generality, in the lemma we consider $\e\in (0,1)$. The lemma finishes the proof of Theorem \ref{thm}.
\begin{lemma}\label{lemmaA} Fix $1>\e>0$ and $n$ such that  $n>8{\ln(2/p)\over \e^2}$. Let $r_n(\a)$  be  the solution of the equality $k(\a,\e,r)={\ln(2/p)\over n}$.
Then $r_n(\a)$ is continuous strictly positive function on $(0,2]$ and $\lim_{\a\to 0}r_n(\a)>0$. In particular,  $\br=\inf_{\a\in (0,2]}r_n(\a)>0$.\end{lemma}
\begin{proof}
For every $\a>0$ and $\e>0$, define  function
$$h(r,x)={(\e e^{-r}-x)\over 1+ {1\over 2} ({4Lr\over x})^{1\over \a}},\quad x\in (0,\e e^{-r}],\quad  r>0.$$
Let us fix $r>0$ and denote $c=4Lr$.  Let
$$x'(\e,r)=\arg\max_{0\leq x\leq \e e^{-r}}{(\e e^{-r}-x)\over 1+{1\over 2}\big({c\over x}\big)^{1\over \alpha}}=\arg\max_{0\leq x\leq \e e^{-r}} h(x,r).$$
Since $x\mapsto h(x,r)$ is continuous and strictly decreasing function, the maximizer $x'$ exists and is unique. It is not difficult to see that $x'$ must satisfy the following equalities:
\begin{equation*}\label{x}
{(\e e^{-r}-x')\over 1+ {1\over 2} ({c\over x'})^{1\over \a}}=\e e^{-r}-(1+\a)x'\quad \Leftrightarrow\quad {2\a\over c^{1\over \a}} ({x'})^{1+{1\over \a}}={\e e^{-r}}-({1+\a})x'.
\end{equation*}
The inequality in the left implies that
\begin{equation}\label{k3}
k(\a,\e,r)={1\over 8}\Big[\sup_{x\in (0,\e e^{-r}]}h(x,r)\Big]^2={1\over 8}\big(\e e^{-r}-(1+\a)x'\big)^2.\end{equation}
The equality in the right implies $x'<{\e e^{-r}\over 1+\a}$ and is equivalent to $x'\big({2\a\over c^{1\over \a}}({x'})^{1\over \a}+(1+\a)\big)=\e e^{-r}$. Hence we obtain
\begin{equation}\label{xbound}
\e e^{-r} \Big( {2\a\over c^{1\over \a}}\left({\e e^{-r} \over 1+\a}\right)^{1\over \a}+(1+\a)\Big)^{-1}< x'<{\e e^{-r}\over 1+\a}.\end{equation}
Define  function
\begin{align}\label{fx}
f(x)=x\Big({2\a\Big({x\over c}\Big)^{1\over \a}}+(1+\a)\Big),\quad x\in (0,\e e^{-r}].
\end{align}
Thus $x'$ is the solution of the equality $f(x)=\e e^{-r}$.\\\\
Suppose now $\a_m\to \a_o>0$. Let $f_m$ and $f_o$ be as $f$ with $\a_m$ and $\a_o$ instead of $\a$, respectively
and let $x_m$ and $x_o$ be the solutions of the equalities $f_m(x)=\e e^{-r}$ and $f_o(x)=\e e^{-r}$. Clearly, for any $x\in (0,\e e^{-r}]$, it holds $f_m(x)\to f_o(x)$. However, since $\a_o>0$, by (\ref{xbound}) we see that there exists $y>0$ so that $x_m\in (y,\e e^{-r}]$ eventually since obviously
  $\sup_{x\in [y,\e e^{-r}]}|f_m(x)-f(x)|\to 0$, we obtain that $x_m\to x_o$ and by (\ref{k3}), thus $k(\a_m,\e,r)\to k(\a_0,\e,r)$. Now observe that for any $\a>0$ and $\e>0$,
 $\lim_{r\to 0}k(\a,\e,r)={\e^2\over 8}$ and $\lim_{r\to \infty}k(\a,\e,r)=0$. Moreover $r\mapsto k(\a,\e,r)$ is strictly decreasing  function. For such functions, pointwise convergence implies uniform convergence, so that as $m$ grows
\begin{equation}\label{yk}\sup_{r\geq 0}|k(\a_m,\e,r)-k(\a_0,\e,r)|\to 0.\end{equation}
The uniform convergence implies that the solutions of the equalities $k(\a_m,\e,r)={\ln(2/p)/n}$ converge as well, i.e. $r_n(\a_m)\to r_n(\a_o)$, provided $n>8{\ln(2/p)\over \e^2}$ so that the  solutions exists.\\\\
We have  proven that the function $r_n(\alpha)$ is continuous on the set $(0,2]$. Let us now consider the case $\a_m\to 0$. From the equalities $f_m(x_m)=\e e^{-r}$, it follows that
$$\lim_m x_m=\left\{
               \begin{array}{ll}
                 c, & \hbox{when $c\leq \e e^{-r}$;} \\
                 \e e^{-r}, & \hbox{when $c>\e e^{-r}$,}
               \end{array}
             \right.
$$
{where $c=4Lr$.} Therefore, for every fixed $r>0$,
$$k(\a_m,\e,r)={1\over 8}\big(\e e^{-r}-(1+\a_m)x_m)^2\to k(0,\e,r)$$
                                                with  $$k(0,\e,r)=\left\{
                                                  \begin{array}{ll}
                                                    {1\over 8}(c-\e e^{-r})^2, & \hbox{if $c\leq \e e^{-r}\quad \Leftrightarrow \quad r\leq W({\e\over 4L})$;} \\
                                                    0, & \hbox{else,}
                                                  \end{array}
                                                \right.$$
where $W$ stands for Lambert $W$-function. Observe that  $r\mapsto k(0,\e,r)$ is strictly decreasing function with limits $\lim_{r\to 0}k(0,\e,r)={\e^2\over 8}$ and $\lim_{r\to \infty}k(0,\e,r)=0$. Hence (\ref{yk}) holds with $\a_o=0$. This, in turn, implies that
$r_n(\a_m)\to r_n(0)$, where $r_n(0)$ is the solution of the equality  $k(0,\e,r_n(0))={\ln(2/p)\over n}$. Since $\lim_{r\to 0}k(0,\e,r)={\e^2\over 8}$, and by assumption ${\e^2\over 8}>{\ln (2/p)\over n}$, we see that
$r_n(0)>0$.  Hence $r_n(\a)$ is continuous strictly positive function on $[0,2]$, thus $\inf_{\a\in (0,2]} r_n(\a)>0$.
\end{proof}
\subsection{Bounding $\ln r$}
We are now interested in finding the probabilistic bounds on difference $|\ln r(u)-\ln \hat{r}(u)|$, where, as previously, $X_1,...,X_n$ is iid sample from symmetric  stable law, ${r}(u)=-\ln |\varphi(u)|=(\gamma|u|)^\alpha$ and $\hat{r}(u)=-\ln |\varphi_n(u)|$. For that an additional assumption has to be made. In the present subsection, we assume that there exists a known lower bound $0<\ua$ such that the parameter space is $[\ua,2]$ instead of $(0,2]$. The crucial benefit of knowing the lower bound of $\ua$ is the fact that the bound $\br$ from Theorem \ref{thm} increases to infinity as $n$ grows.
Recall that the bound $\br$ in lacks this property: although $\br$ increases with $n$, it always satisfies $\br<W({\e\over 2L(1+\e)})<1$ and, as argued in Introduction, such a property might be unavoidable. Bounding the parameter space away from zero, we have a new bound that tends to infinity as $n$ increases.
Let us formulate it as a corollary.
\begin{corollary}\label{kolla}
Assume $\ua>0$ to be given.  Fix $\e>0$, $p\in (0,1)$ and let $n_o(\e,p)<\infty$ be as in Theorem \ref{thm}. Then
  for every $n>n_o$ there exists $\bar{r}(\e,p,n,\ua)$  independent of $\a$ and $\gamma$ such that $\lim_n \bar{r}(\e,p,n,\ua)=\infty$ and
$$P\Big( \sup_{u>0: r(u)\leq \overline{r}}|r(u)-\hat{r}(u)|>\e\Big)\leq p,$$
for every $\alpha\in [\ua,2]$ and $\gamma>0$.\end{corollary}
\begin{proof}
Fix $0<\ua<2$ and take $r_o$ so big that $4Lr_o>\e e^{-r_o}$. Then for every $r\geq r_o$ and for every $x\in[0,\e e^{-r}]$, the function
$$\a \mapsto {(\e e^{-r}-x)\over 1+ {1\over 2} ({c4Lr\over x})^{1\over \a}}$$
is increasing in $\a$. Therefore
\begin{equation}\label{vrr}
k(\a,\e,r)\geq k(\ua,\e,r),\quad \forall \a\geq \ua\quad r\geq r_o.
\end{equation}
Now observe that when $\a>0$, then for every $r$, it holds $x'< {\e e^{-r}\over 1+\a}$, thus $k(\a,\e,r)={1\over 8}\big(\e e^{-r}-(1+\a)x')^2>0$. Since $\lim_{r\to \infty}k(\a,\e,r)=0$, it follows that $\lim_n r_n(\a)=\infty$.
Therefore, there exists $n(r_o,\ua)$ so big that $r_n(\ua)>r_o$.  From (\ref{vrr}), it follows that when $n>n_o$, it holds  $r_n(\a)\geq r_n(\ua)$, $\forall \a\geq \ua$. Hence $\lim_n \inf_{\a\in [\ua,2]}r_n(\a)=\lim_n r_n(\ua)=\infty.$ The inequality in the statement now follows from Theorem \ref{thm}.
\end{proof}
\begin{theorem}\label{thm2} Assume $\ua>0$ to be given.  Fix $\e>0$, $p\in (0,1)$ and $\ur>0$.
Take $n_1(\e,p,\ua,\ur)$ as minimal $n$ such that
\begin{equation}\label{ass12}\bar{r}({\e\over 1+\e}\ur,p,n,\ua)>\ur,\quad n_1(\e,p,\ua,\ur)\geq n_o({\e\over 1+\e}\ur,p),\end{equation}
where $n_o(\e,p)$ is as in Theorem \ref{thm} and the function $\bar{r}$ is as in Corollary \ref{kolla}. Then for every $n\geq n_1$, $\alpha\in [\ua,2]$ and $\gamma>0$
\begin{equation}\label{ln-vorr}
P\Big( \sup_{u: r(u)\in [\ur, \overline{r}]}|\ln r(u)-\ln \hat{r}(u)|>\e\Big)\leq p,\end{equation}
where $\bar{r}=\bar{r}({\e\over 1+\e}\ur,p,n,\ua)>\ur$.\end{theorem}
\begin{proof}  By Corollary \ref{kolla}, such a finite $n_1$ exists. Now, for any $n>n_1$, by Corollary \ref{kolla} again,
$$
P\Big( \sup_{u: r(u) \in[\underline{r},\bar{r}]}|\ln r(u)-\ln
\hat{r}(u)|>\e\Big)\leq P\Big( \sup_{u: r(u)
\in[\underline{r},\bar{r}]}|r(u)-\hat{r}(u)|>{\e\over
1+\e}\underline{r}\Big)\leq p.$$
\end{proof}
\section{Applications of Theorem \ref{thm2}}\label{sec:hinnang}
Recall the estimate $\hat{\alpha}$ in (\ref{hinn}). The construction of $\hat{\alpha}$ requires fixing the constants  $\ur,\br$ and $\epsilon$. The following simple lemma shows how Theorem \ref{thm2} can be used to prove a large deviation inequality for $\hat{\a}$.
\begin{lemma}\label{lemma:prec} Suppose $\ur<\br$ and $\epsilon$ are chosen such that  $\br-2\e>\ur+2\e$ and (\ref{ln-vorr}) holds for some $p\in (0,1)$. Take $\e_1>0$ so big that
\begin{equation}\label{suur0}{\e_1\over 4}\ln \left({\br-2\e\over \ur+2\e}\right)\geq \e.\end{equation}
Then
\begin{equation}\label{tapsus}
P\Big(|\hat{\a}-\a|\geq \e_1\Big)\leq p.
\end{equation}
\end{lemma}
\begin{proof}
Since
 $u\mapsto \hat{r}(u)$ is continuous, we have by (\ref{hinn}) that $\hat{r}(u_1)=\br-\e$ and $\hat{r}(u_2)=\ur+\e$.
Let
\begin{equation}\label{sisu}
E=\big\{\sup_{u: r(u)\in [\ur, \br]}|\ln r(u)-\ln \hat{r}(u)|\leq \e\big\}.\end{equation}
On the event $E$, it holds
$|\hat{r}(u_i)-r(u_i)|\leq \e$ for $i=1,2$ and so
 $\br-2\e\leq r(u_1)\leq  \br$ and $\ur \leq r(u_2)\leq \ur+2\e$.
Therefore, on the set $E$, for any  $\e_1>0$ the following implications hold
\begin{align*}
\{|\hat{\a}-\a|\geq \e_1\}&=\Big\{\Big|{\ln (\br-\e)-\ln  (\ur+\e) \over \ln u_1-\ln u_2}-{\ln r(u_1)-\ln r(u_2) \over \ln u_1-\ln u_2}\Big|> \e_1\Big\}\\
&=\big\{|(\ln \hat{r}(u_1)-\ln r(u_1))+(\ln r(u_2)-\ln \hat{r}(u_2))|> {\e_1  \ln (u_1/u_2)}\big\}\\
&\subseteq \big\{\sup_{u: r(u)\in [\ur, \br]} |\ln r(u)-\ln \hat{r}(u)|> {\e_1\over 2}\ln (u_1/u_2)\big\}\\
&=\big\{\sup_{u: r(u)\in [\ur, \br]} |\ln r(u)-\ln \hat{r}(u)|> {\e_1\over 2\a}\ln (r(u_1)/r(u_2))\big\}\\
&\subseteq \Big\{\sup_{u: r(u)\in [\ur, \br]} |\ln r(u)-\ln \hat{r}(u)|> {\e_1\over 4}\ln \left({\br-2\e\over \ur+2\e}\right)\Big\}.
\end{align*}

We have thus shown that
$$E\cap\{|\hat{\a}-\a|\geq \e_1\} \subseteq \Big\{\sup_{u: r(u)\in [\ur, \br]} |\ln r(u)-\ln \hat{r}(u)|> {\e_1\over 4}\ln \big({\br-2\e\over \ur+2\e}\big)\Big\}.$$

Taking now $\epsilon_1$ so big that  (\ref{suur0}) holds, we obtain that
$$E\cap\{|\hat{\a}-\a|\geq \e_1\} \subseteq  E^c$$
which obviously implies that $\{|\hat{\a}-\a|\geq \e_1\} \subseteq  E^c$ and so (\ref{tapsus}) holds.
\end{proof}
\paragraph{Exact confidence intervals.}
In what follows, we shall briefly discuss how to choose $\e>0,\ur, \br$ such that (\ref{tapsus}) holds. In particular, we shall address the following classical problems of parameter estimation:
\begin{description}
  \item[Q1:] Given lower bound $\ua>0$, precision $\e_1$ and probability $p>0$, find possibly small sample size $n$ and  constants  $\ur,\br,\e$ (for constructing $\hat{\a}$) so that the estimate (\ref{hinn}) satisfies inequality (\ref{tapsus}).
  \item[Q2:] Given lower bound $\ua>0$, sample size $n$ and probability $p>0$, find possibly small $\e_1>0$ and $\ur,\br,\e$ (for constructing $\hat{\a}$) so that the estimate (\ref{hinn}) satisfies inequality (\ref{tapsus}). In other words, find exact (non-asymptotic) confidence intervals: with probability $1-p$: $\hat{\a}-\e_1\leq \a \leq \hat{\a}+\e_1$.
\end{description}
To solve {\bf Q1}, define for any $\rho>0$, $n$ and $\e>0$
\begin{equation}
F(n,\rho,\e):={\bar{r}\big(\rho,p,n,\ua\big)-2\e\over \rho{1+\e\over \e}+2\e}-\exp[{4\over \e_1}\e],\end{equation}
where $\bar{r}\big(\rho,p,n,\ua\big)$ is as in Corollary \ref{kolla}. The function $F$ depends also on $\ua,\e_1$ and $p$, but these parameters are fixed and left out from notation. Define
$$F(n,\rho)=\sup_{\e>0}F(n,\rho,\e),\quad F(n)=\sup_{\rho>0}\Big[F(n,\rho) \wedge (n-n_o(\rho,p))\Big],$$
where $n_o(\rho,p)$ is as in Theorem \ref{thm}. Now  take $n_1$ minimal $n$ such that $F(n_1)>0$. Observe that $F(n)$ is increasing, so that when $F(n_1)>0$, then $F(n)>0$ for any $n>n_1$. The estimation procedure is now the following.
\paragraph{EstimationProcedure1:}
\begin{enumerate}
  \item Find $n_1$ such that $F(n_1)>0$.
  \item Given $n\geq n_1$ find $\rho>0$ such that $F(n,\rho)>0$.
  \item Use $\rho$  to find $\e>0$ such that $F(n,\rho_o,\e)>0$.
  \item Use   $\rho$ to determine $\br=\br(\rho,p,n,\ua)$, where $\br(\rho,p,n,\ua)$ is as in Corollary \ref{kolla}.
  \item Use $\e$ and $\rho$ to find $\ur=\rho{1+\e\over \e}$.
  \item Use $\ur$ and $\br$ to find $u_1$ and $u_2$ as in (\ref{u}).
  \item Estimate $\hat{r}(u)$ based on the sample of size $n$ and the estimate of characteristic function.
  \item Find $\hat{\a}$ as in (\ref{hinn}).
\end{enumerate}
\begin{theorem}\label{thm21} Let $\ua$, $\e_1>0$ and $p\in (0,1)$ be given. Let  $n$ be the sample size satisfying $F(n)>0$. Then the estimate $\hat{\a}$ obtained via {\rm EstimationProcedure1}  satisfies the inequality (\ref{tapsus}), provided the true parameter $\a$ satisfies the inequality $\a\geq \ua$.\end{theorem}
\begin{proof} According to definition of $F(n)$, the parameters $\rho>0$ and $\e>0$ as specified by {\it EstimationProcedure1}  are such that $F(n,\rho,\e)>0$ and $n>n_o(\rho,p)$. With $\ur=\rho{1+\e\over \e}$, we see that
$${\br-2\e\over \ur+2\e}>\exp[{4\over \e_1}\e]$$ so that the
equation (\ref{suur0}) holds. This equation also implies that $\br>\ur+4\e$. Since $n>n_o(\rho,p)$, we have $n>n_o({\e\over 1+\e}\ur,p)$. Thus both inequalities in (\ref{ass12}) hold and therefore  $n\geq n_1$, where $n_1(\e,p,\a,\ur)$ is as in Theorem \ref{thm2}. Hence  all assumptions of Theorem \ref{thm2} are fulfilled and so (\ref{ln-vorr}) holds. Both assumptions of Lemma \ref{lemma:prec} are fulfilled and so the inequality
 (\ref{tapsus}) holds as well.\end{proof}

To solve {\bf Q2},  we need to assume some minimal requirements about the given sample size $n$. In what follows, we assume that there exists $\rho_o>0$ such that $n>n_o(\rho_o,p)$, where $n_o(\rho,p)$ is as in Theorem \ref{thm}. Now we  define
 $$F(\e_1,\rho,\e):={\bar{r}\big(\rho,p,n,\ua\big)-2\e\over \rho {1+\e\over \e}+2\e}-\exp[{4\over \e_1}\e].$$
The function $F(\e_1,\rho,\e)$ also depends on $n$ and $p$, but these are fixed. As previously, define
$$F(\e_1,\rho):=\sup_{\e>0}F(\e_1,\rho,\e),\quad F(\e_1):=\sup_{\rho\geq \rho_o}F(\e_1,\rho).$$
 Now find (as small as possible) $\e_1>0$ such that $F(\e_1)>0$.
\\\\
{\bf EstimationProcedure2:}
\begin{enumerate}
  \item Given $\e_1>0$ satisfying $F(\e_1)>0$ find  $\rho\geq \rho_o$ such that  $F(\e_1,\rho)> 0$.
   \item Use $\rho$ to find $\e>0$ such that $F(\e_1,\rho,\e)>0$
  \item Use $\rho$ to determine $\ur=\ur(\rho,p,n)$, where $\br(\rho,p,n,\ua)$ is as in Corollary \ref{kolla}.
  \item Use $\e$ and $\rho$ to find  $\ur=\rho{1+\e\over \e}$.
\item Use $\ur$ and $\br$ to find $u_1$ and $u_2$ as in (\ref{u}).
  \item Estimate $\hat{r}(u)$ based on the sample of size $n$ and the estimate of characteristic function.
  \item Find $\hat{\a}$ as in (\ref{hinn}).
\end{enumerate}
\begin{theorem}\label{thm22} Let  $\ua>0$, the sample size $n>n_o(\rho_o,p)$, where $\rho_o>0$ and $p\in (0,1)$ be given.
 Let  $\e_1$  satisfy $F(\e_1)> 0$. Then the estimate $\hat{\a}$ obtained via {\rm EstimationProcedure2}  satisfies the inequality (\ref{tapsus}), provided the true parameter $\a$ satisfies the inequality $\a\geq \ua$.\end{theorem}
\begin{proof} According to definition of $F(\e_1)$, the parameters $\rho>0$ and $\e>0$ as specified by {\it EstimationProcedure2} are such that $F(\e_1,\rho,\e)>0$. Since $\rho \mapsto n_o(\rho,p)$ is decreasing, it holds that   $n>n_o(\rho,p)$. With $\ur=\rho{1+\e\over \e}$, we see that  equation (\ref{suur0}) holds. This equation also implies that $\br>\ur+4\e$ and hence  both inequalities in (\ref{ass12}) hold and therefore  $n\geq n_1$, where $n_1(\e,p,\ua,\ur)$ is as in Theorem \ref{thm2}. Hence  all assumptions of Theorem \ref{thm2} are fulfilled and so (\ref{ln-vorr}) holds. Then Lemma \ref{lemma:prec} implies that
 (\ref{tapsus}) holds as well.\end{proof}
\section{An alternative construction}\label{sec:alternative}
Recall the construction of the  $\br$ in the proof of Theorem \ref{thm}: the key of the construction is the large deviation inequality
\begin{equation}\label{ineq}
P\Big(e^{\br}\sup_{u\in (0,\u]}|\v_n(u)-\v(u)|>\e\Big)\leq A\exp[-n\cdot k(\a,\e,\br)],
\end{equation}
where $A=2$ and $k(\a,\e,\br)>0$. Then $r_n(\a)$ was defined as the solution of the equality
\begin{equation}\label{eq}
A\exp[-n\cdot k(\a,{\e\over 1+\e},r_n)]=p\end{equation}
and so the the desired bound $\br=\inf_{\a \in (0,\a]}r_n(\a)$ was obtained. In order (\ref{eq}) to have positive solution, the sample size $n$ must satisfy $n>n_o(\e,p)$.
The large deviation bound (\ref{eq}) constructed in Section \ref{subsec1:main} is not the one possible option. We now sketch another possible construction yielding to a different inequality, and therefore, also to the different function $r_n(\a)$ and different bound $\br$.
Unlike the $\br$ obtained in Section \ref{subsec1:main}, the new  $\br$ has  more explicit form. The price for it is much bigger required sample sice $n_o$.
\\\\
Observe that by (\ref{3osa2}) and (\ref{Ka}) the inequality
$\|F_n-F\|\leq {\delta\over 8}$ implies that
 $$|\v_n(u)-\v(u)|\leq \delta+ 2\|F_n-F\| \Big({ 8Lr(u)\over \delta }\Big)^{1\over \alpha}.$$
 Hence for every $\u>0$, \begin{align*}
\{e^{r(\u)}\sup_{u\in(0,\u]}|\varphi(u)-\v_n(u)|>\e\}\subset
\{\|F_n-F\|> {\delta\over 8}\}\cup\Big\{2\|F_n-F\|\Big({
8Lr(\u)\over \delta }\Big)^{1\over \alpha} \geq \e e^{-r(\u)}-\delta
\Big\}.
\end{align*}
{Recall $\be=\e e^{-\br}$ and $\br=r(\u)$.} Hence, we obtain
\begin{align}\nonumber
P\Big(e^{\br}\sup_{u\in (0,\u]}|\v_n(u)-\v(u)|>\e\Big)&\leq
P\Big(\|F_n-F\|> {\delta\over 8}\Big)+P\Big(\|F_n-F\|\geq {(\be
-\delta)\delta^{1\over \a}\over 2( 8L\br )^{1\over
\alpha}}\Big).
\\\label{summa}
 & \leq 2\exp[-n\cdot {\delta^2\over
32}]+2\exp\Big[-n\cdot{(\be -\delta)^2\delta^{2\over \a}\over
2( 8L\br )^{2\over \alpha}}\Big].
\end{align}
Choose
$$\delta={\be\over 1+\a}=\arg\max_{\delta \in [0,\be]}(\be-\delta)^2\delta^{2\over \alpha}.$$
Thus plugging $\delta$ into (\ref{summa}), we obtain with $ s=2(1+{1\over \alpha})$,
\begin{align*}
P\Big(e^{\br}\sup_{u\in (0,\u]}|\v_n(u)-\v(u)|>\e\Big)& \leq
2\exp[-n\cdot {\be^2\over 32(1+\a)^2}]+2\exp\Big[-n\cdot
\Big(\big({\be\over 1+\a}\big)^s{\alpha^2\over
2(8L\br)^{2\over \alpha}}\Big)\Big]\\
&\leq 4\exp[-n\cdot k(\a,\e,\br)],
\end{align*}
where
$$k(\a,\e,\bar{r})=\big({\e^2e^{-2\bar{r}}\over
32(1+\a)^2}\big)\wedge \Big({\e^s e^{-s\bar{r}}\over 2(1+\a)^s}{\alpha^2\over
(\bar{r}8L)^{2\over \alpha}}\Big).$$
The inequality (\ref{eq}) holds with an equality, when
 $${\bar r}=r_{1,n}(\a)\wedge r_{2,n}(\a)$$
and $r_{1,n}$, $r_{2,n}$ are solutions of the following equalities:
\begin{align}\label{v1}
\exp[-2r_{1,n}]&={\ln (4/p)\over n}\big({1+\e\over \e}\big)^2 32(1+\a)^2
\\\label{v2} {d(\a) \exp[-s r_{2,n}]\over
(r_{2,n})^{2\over \alpha}}&={\ln (4/p)\over n}\big({1+\e\over
\e}\big)^s,\quad \text{where  } d(\a)={\a^2\over 2(1+\a)^s}{1\over
(8L)^{2\over \a}}
\end{align}
Thus
$$r_{1,n}(\a)=-{1\over 2} \ln \Big[{\ln (4/p)\over n}\big({1+\e\over \e}\big)^2
32(1+\a)^2\Big]$$
and $r_{1,n}(\a)>0$ holds when
\begin{equation}\label{tingimus}
n>{\ln (4/p)}\big({1+\e\over \e}\big)^2 32(1+\a)^2.
\end{equation}
From (\ref{tingimus}), we obtain necessary sample size
\begin{equation}\label{tingimus2}
n_o(\e,p)={\ln (4/p)}\big({1+\e\over \e}\big)^2\cdot 288.
\end{equation}
The solution of equality (\ref{v2}) is
$$r_{2,n}(\a)={1\over \a+1}W((\a+1)g_n(\a)),$$
where $W$ is Lambert's W-function and $g_n$ is defined as follows
$$g_n(\a)={\a^{\a}\over 2^{\a\over 2}(\a+1)^{\a+1}8L}\Big({n\over \ln (4/p)}\Big)^{\a\over 2}\Big({\e\over 1+\e}\Big)^{\a+1}.$$
Observe that  $\lim_{\a\to 0} g_n(\a)={1\over 8L}{\e\over 1+\e}=:g_o$ and $\lim_{\a\to 0}r_{2,n}(\a)=W(g_o)\in (0,1).$ It can be shown that
\begin{equation}\label{ur}
\br(\e,p,n)=\min_{\a}\{r_{1,n}(\a),r_{2,n}(\a)\}=\ln {\kappa}-\ln \Big(12 \vee {\text{{${1+\a_2\over \a_2}$}}}\Big)>0,
\end{equation}
where
$\kappa=\frac{\e}{\e+1}\sqrt{{\frac{n}{2\ln(4/p)}}}$
while {${r}_{1,n}>0$ only if $\kappa>12$ and $n>n_0$, where $n_o$ is given by \eqref{tingimus2}, and  $\a_2\in(0,2]$} is the solution of the equality
$r_{2,n}(\a)=\ln \kappa\frac{\a}{1+\a}.$

Since $\lim_{\a\to 0}r_{2,n}(\a)=W(g_o)$, it holds that for any $n$   $\br<1$, and so the alternative construction satisfies our conjecture. However, it is possible to show the for any $\ua>0$,
$$\bar{r}(\e,p,n,\ua)=\inf_{\a\in (\ua,2]}r_n(\a)\to \infty.$$
So we have another construction that  confirms our conjecture that the universal bound $\br$ satisfies the inequality $\br<1$ and the universal bound $\u$ does not exists.
\section{Comparison of $r_n$ and $\br$}
Fix $\e=0.1$, $p=0.1$ and $L=1/2$. We  find the  values of $r_n(\a)$  obtained in Section \ref{subsec1:main}  as follows: first we numerically  find $x'$ as the solution of the equality $f(x)=e\mathrm{e}^{-r}$, where $f(x)$ is given by \eqref{fx}, then we calculate $k(\a,\e,r)$ given by \eqref{k3} and then $r_n$ can be found as solution of \eqref{r3def}. We compare the obtained values of $r_n(\a)$ with the ones obtained in Section \ref{sec:alternative}: $r_n(\a)=\min\{r_{1,n}(\a),r_{2,n}(\a)\}$, where  $r_{1,n}(\a)$  is the solution of equality (\ref{v1}) and $r_{2,n}(\a)$  is the solution of equality (\ref{v2}). Note that by  \eqref{tingimus2} we have ${r}_{1,n}(\a)>0$ only if  $n>n_o\approx 128550$.
Figure \ref{fig4} plots both constructions of $r_n$ versus $\a\in(0.01, 2]$ for different sample sizes $n$.
\begin{figure}[!htp] \centering{\includegraphics[height=0.25\textheight, width=1\textwidth]{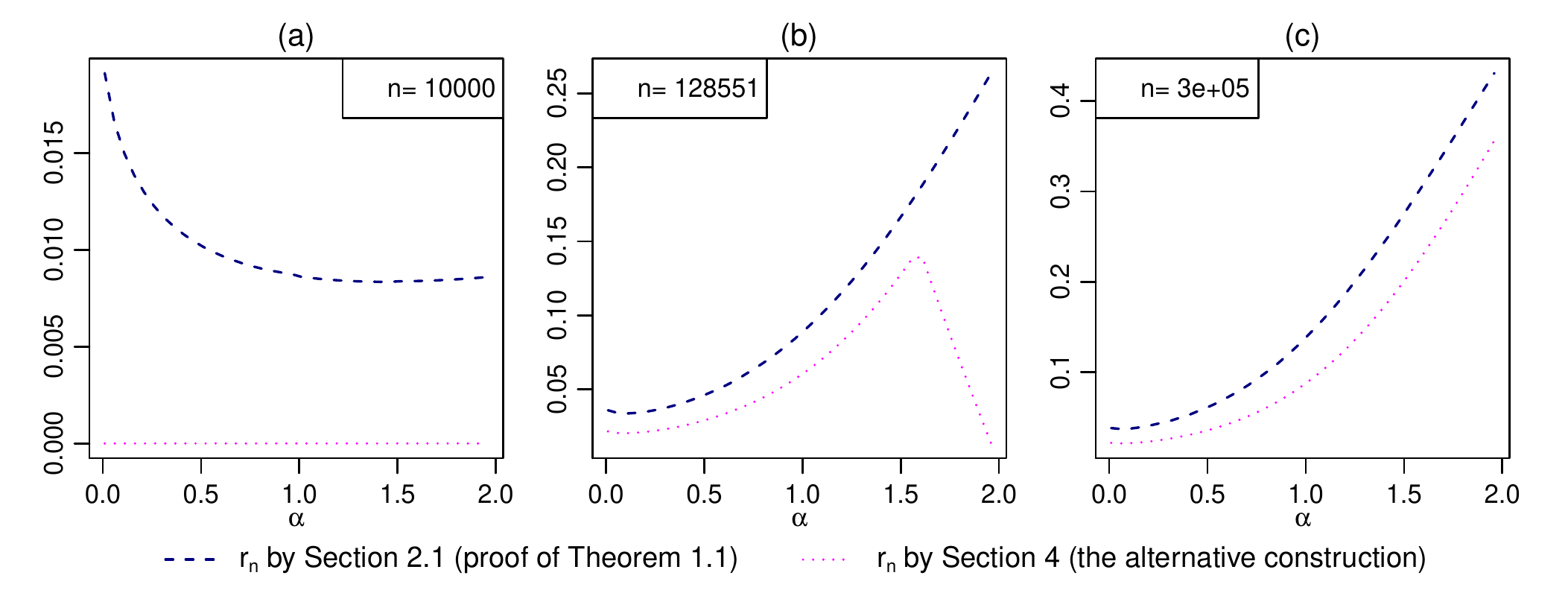}}\caption{The values of $r_n$ vs $\alpha$ with $\e=0.1$, $p=0.1$ and $n=1000$ in (a), $n=128551>n_o$ in (b) and $n=300000$ in (c).}\label{fig4}
\end{figure}Obviously, in Figure \ref{fig4} (a) the $r_n(a)$ obtained by the alternative construction in Section \ref{sec:alternative} is $0$ because $n<n_o$.  However, the values of $r_n$ obtained by  Section \ref{subsec1:main} are small but positive.
\begin{figure}[!htp] \centering{\includegraphics[height=0.25\textheight, width=1\textwidth]{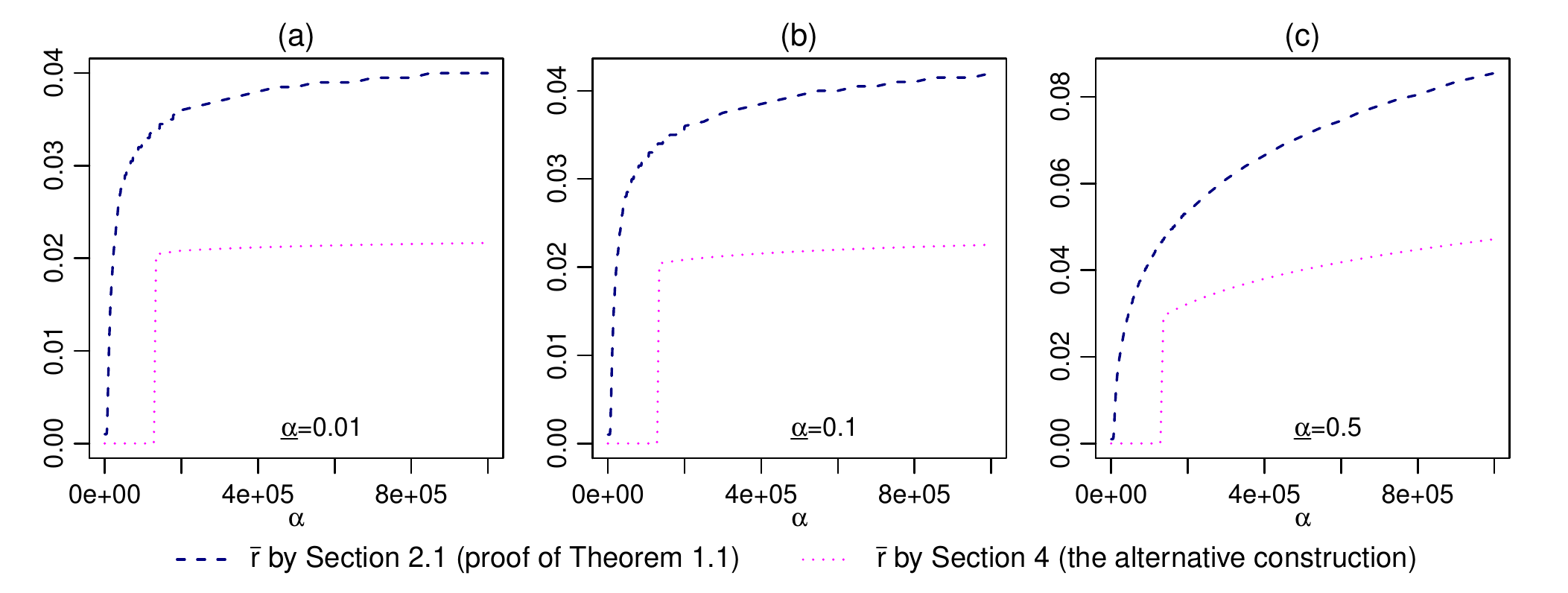}}\caption{The values of $\bar{r}$ vs $n$ with $\e=0.1$, $p=0.1$ and  $\ua=0.01$ in (a), $\ua=0.1$ in (b) and $\ua=0.5$ in (c).}\label{fig5}
\end{figure}
  In Figure \ref{fig4} (b) both constructions give positive results (because $n=n_o+1$) while  alternative construction gives smaller (more conservative) values, with drop after $\a=1.5$ (and minimum over $\a$ is at $\a=2$). In Figure \ref{fig4} (c) the large sample size such as $n=3\cdot 10^5$ is used and both $r_n(\a)$ behave in concordance while the alternative construction yields smaller values for all values of $\a$. Next we  compare the values of $\br=\min_{\a\in(\ua,2]}r_n(\a)$ obtained by Section \ref{subsec1:main} and Section \ref{sec:alternative}.  Figure \ref{fig5} plots $\br$  versus $n$ for   different lower limits $\ua$.  It is clearly evident from Figure \ref{fig5} that $\br$ is increasing in $n$ (while alternative construction requires $n>n_o$). Setting lower limit from  $\ua=0.1$ to  $\ua=0.5$ increases the  values of $\br$ approximately twice. All in all, in our example the construction given by  Section \ref{subsec1:main} yields much bigger (less conservative) values of $\br$ than the more explicit form construction given by Section \ref{sec:alternative}.

\bibliography{References_Krutto}

\newcommand{\noop}[1]{}
\begin{thebibliography}{}

\bibitem[\protect\citeauthoryear{Bibalan, Amindavar, and
  Amirmazlaghani}{Bibalan et~al.}{2017}]{Bibalan2017CharacteristicFB}
Bibalan, M.~H., H.~Amindavar, and M.~Amirmazlaghani (2017).
\newblock Characteristic function based parameter estimation of skewed
  alpha-stable distribution: An analytical approach.
\newblock {\em Signal Process.\/}~{\em 130}, 323--336.

\bibitem[\protect\citeauthoryear{Cs\"{o}rgo}{Cs\"{o}rgo}{1981}]{Csorgo1981}
Cs\"{o}rgo, S. (1981).
\newblock Limit behaviour of the empirical characteristic function.
\newblock {\em The Annals of Probability\/}~{\em 9\/}(1), 130--144.

\bibitem[\protect\citeauthoryear{Kakinaka and Umeno}{Kakinaka and
  Umeno}{2020}]{kakinaka2020flexible}
Kakinaka, S. and K.~Umeno (2020).
\newblock Flexible two-point selection approach for characteristic
  function-based parameter estimation of stable laws.
\newblock Online at \url{https://arxiv.org/abs/2005.11499}.

\bibitem[\protect\citeauthoryear{Kogon and Williams}{Kogon and
  Williams}{1998}]{Kogon1998}
Kogon, S.~M. and D.~B. Williams (1998).
\newblock Characteristic function based estimation of stable distribution
  parameters.
\newblock In R.~J. Adler, R.~E. Feldman, and M.~S. Taqqu (Eds.), {\em A
  Practical Guide to Heavy Tails}, pp.\  311--335. Boston: Birkh\"{a}user.

\bibitem[\protect\citeauthoryear{Krutto}{Krutto}{2016}]{Krutto2016}
Krutto, A. (2016).
\newblock Parameter estimation in stable law.
\newblock {\em Risks\/}~{\em 4\/}(4), 43.

\bibitem[\protect\citeauthoryear{Krutto}{Krutto}{2018}]{Krutto2018}
Krutto, A. (2018).
\newblock Empirical cumulant function based parameter estimation in stable
  laws.
\newblock {\em Acta et Commentationes Universitatis Tartuensis de
  Mathematica\/}~{\em 22\/}(2), 311--338.

\bibitem[\protect\citeauthoryear{McCulloch}{McCulloch}{1996}]{McCulloch1986}
McCulloch, J.~H. (1996).
\newblock Financial applications of stable distributions.
\newblock In G.~Maddala and C.~Rao (Eds.), {\em Statistical Methods in
  Finance}, Volume~14 of {\em Handbook of Statistics}, pp.\  393 -- 425.
  Elsevier.

\bibitem[\protect\citeauthoryear{Nolan}{Nolan}{2001}]{Nolan2001}
Nolan, J.~P. (2001).
\newblock Maximum likelihood estimation and diagnostics for stable
  distributions.
\newblock In O.~Barndorff-Nielsen, S.~Resnick, and T.~Mikosch (Eds.), {\em
  L{\'e}vy Processes}, pp.\  379--400. Boston: Birkh\"{a}user.

\bibitem[\protect\citeauthoryear{Nolan}{Nolan}{2018}]{Nolan2018}
Nolan, J.~P. (2018).
\newblock {\em Stable Distributions - Models for Heavy Tailed Data}.
\newblock Boston: Birkh\"{a}user.
\newblock In progress, Chapter 1 online at
  \url{http://fs2.american.edu/jpnolan/www/stable/stable.html}.

\bibitem[\protect\citeauthoryear{Paulson, Holcomb, and Leitch}{Paulson
  et~al.}{1975}]{Paulson1975}
Paulson, A.~S., E.~W. Holcomb, and R.~A. Leitch (1975).
\newblock The estimation of the parameters of the stable laws.
\newblock {\em Biometrika\/}~{\em 62\/}(1), 163--170.

\bibitem[\protect\citeauthoryear{Press}{Press}{1972}]{Press1972}
Press, J.~S. (1972).
\newblock Estimation in univariate and multivariate stable distributions.
\newblock {\em J. Amer. Statist. Assoc.\/}~{\em 67\/}(340), 842--846.

\bibitem[\protect\citeauthoryear{{Robust Analysis Inc.}}{{Robust Analysis
  Inc.}}{2017}]{STABLER}
{Robust Analysis Inc.} (2017).
\newblock {\em STABLE 5.3 R Version for Windows}.
\newblock Washington, DC, USA: {Robust Analysis Inc.}
\newblock \url{http://www.robustanalysis.com/}.

\bibitem[\protect\citeauthoryear{Samorodnitsky and Taqqu}{Samorodnitsky and
  Taqqu}{1994}]{Samorodnitsky1994}
Samorodnitsky, G. and M.~S. Taqqu (1994).
\newblock {\em Stable Non-Gaussian Random Processes: Stochastic Models with
  Infinite Variance}.
\newblock New York: Chapman \& Hall.

\bibitem[\protect\citeauthoryear{{van der Vaart}}{{van der
  Vaart}}{1998}]{van1998asymptotic}
{van der Vaart}, A. (1998).
\newblock {\em Asymptotic Statistics}.
\newblock Cambridge: Cambridge University Press.

\bibitem[\protect\citeauthoryear{Zolotarev}{Zolotarev}{1986}]{zolotarev1986}
Zolotarev, V. (1986).
\newblock {\em One-dimensional Stable Distributions}, Volume~65 of {\em
  Translations of mathematical monographs}.
\newblock American Mathematical Society.

\end{thebibliography}

\end{document}